\documentclass[12pt]{article}

\RequirePackage{kpfonts}
\RequirePackage[sf,bf,small,raggedright]{titlesec}
\usepackage{textcomp}
\usepackage{amssymb}
\usepackage{bbm}
\usepackage{fancyhdr}
\usepackage{amsmath}
\usepackage{amsbsy,amsthm}
\usepackage{amscd}
\usepackage{latexsym}
\usepackage{graphicx}   % to include and manipulate graphics
\usepackage{pdfsync}
\usepackage{blkarray}
\usepackage{multirow}
\usepackage{hyperref}

\newcommand{\R}{\mathbb{R}}

\newcommand{\eps}{\varepsilon}

\newtheorem{lemma}{Lemma}
\newtheorem{theorem}{Theorem}

\newtheorem{proposition}{Proposition}

\numberwithin{equation}{section}

\newcommand{\EXP}{\mathbb{E}}
\newcommand{\PROB}{\mathbb{P}}

\newcommand{\lambdastar}{\lambda_{*}}
\newcommand{\IND}[1]{\mathbbm{1}_{\{ #1 \}}}

\newcommand{\one}{\mathbbm{1}}
\newcommand{\Id}{\mathrm{Id}}

\begin{document}

\begin{titlepage}
\thispagestyle{empty}
\setcounter{page}{1}

 \title{Local optima of the Sherrington-Kirkpatrick Hamiltonian}
% \thanks{
% G\'abor Lugosi was supported by the Spanish Ministry of Science and Technology grant MTM2015-67304.
% }
% }
\author{
Louigi Addario-Berry
\thanks{Supported by  NSERC Discovery Grant 341845.}
\\
Mathematics and Statistics \\
McGill University, \\
Montreal, Canada \\
louigi.addario@mcgill.ca
  \and 
Luc Devroye
\thanks{Supported by the Natural Sciences and Engineering
Research Council (NSERC) of Canada.}
\\
School of Computer Science \\
McGill University, \\
Montreal, Canada \\
lucdevroye@gmail.com
\and
G\'abor Lugosi
\thanks{
Supported by 
the Spanish Ministry of Economy and Competitiveness,
Grant MTM2015-67304-P and FEDER, EU.
}
 \\
ICREA,  \\ 
Department of Economics,  \\
Pompeu Fabra University, \\
Barcelona GSE \\
%Graduate School of Economics \\
Barcelona, Spain; \\ 
gabor.lugosi@upf.edu
\and 
Roberto Imbuzeiro Oliveira
\thanks{Support from CNPq, Brazil
via {\em Ci\^{e}ncia sem Fronteiras} grant \# 401572/2014-5.
Supported by a {\em Bolsa de Produtividade em Pesquisa} from CNPq, Brazil.
Supported by FAPESP Center for Neuromathematics (grant\# 2013/ 07699-0 , FAPESP - S. Paulo Research Foundation).} \\
IMPA, Rio de Janeiro, RJ, \\
Brazil;\\ 
rimfo@impa.br
}

\maketitle

\begin{abstract}
We study local optima of the Hamiltonian of the Sherrington-Kirkpatrick
model. We compute the exponent of the expected number of local optima 
and determine the ``typical'' value of the Hamiltonian.
\end{abstract}

% \noindent
% {\bf Keywords:}

\end{titlepage}

\setcounter{page}{2}
\section{Local optima of the Hamiltonian}

Let $W=(W_{i,j})_{n\times n}$ be a symmetric matrix with zero diagonal such that the $(W_{i,j})_{1\le i<j\le n}$ 
are independent standard normal random variables.%We assume, for simplicity, that $n$ is even. 
The \emph{Sherrington-Kirpatrick} model of spin glasses is defined by a \emph{random Hamiltonian}, that is, a random function $H:\{-1,+1\}^n\to \R$. For a configuration $\sigma = (\sigma_i)_{i=1}^n\in \{-1,+1\}^n$, $H(\sigma)$ is defined as follows.
\[H(\sigma) := \sum_{1\leq i<j\leq n}\sigma_i\,\sigma_j\,W_{ij}.\]
We follow the usual convention of calling $\sigma\in \{-1,+1\}^n$ a \emph{spin configuration}, the coordinates of $\sigma$ \emph{spins}, and the value $H(\sigma)$ the \emph{energy} of configuration $\sigma$. 

Given $i\in[n]$ and $\sigma$ as above, we let $\sigma^{(i)}$ denote a new configuration obtained from $\sigma$ by flipping the $i$-th spin and leaving other coordinates unchanged. That is,
\[\sigma^{(i)}_j:= \left\{ \begin{array}{ll}-\sigma_i, & j=i;\\ \;\sigma_j, & j\in [n]\backslash \{i\}.\end{array}\right.\] 
We say that $\sigma$ is a \emph{local minimum} or a \emph{local optimum} of $H$ if 
\[\forall i\in[n]\,:\, H(\sigma^{(i)})\geq H(\sigma).\]
That is, $\sigma$ is a local minimum if flipping the sign of any individual spin does not decrease the value of the energy. 

The global optimum $\min_{\sigma\in \{-1,+1\}^n} H(\sigma)$---called the ``ground-state energy''---has been extensively
studied. 
The problem was introduced by Sherrington and Kirkpatrick
\cite{ShKi75}
as a mean-field model for spin glasses. The value of the optimum was determined non-rigorously in the seminal work
of Parisi \cite{Par80}, as a consequence of the so-called ``Parisi
formula''. Parisi's formula was proved by Talagrand \cite{Tal06} in a
breakthrough paper, see also Panchenko \cite{Pan13} for an overview.
It follows from Talagrand's result that
\[
  n^{-3/2} \min_{\sigma\in \{-1,+1\}^n} H(\sigma) \to - c \quad \text{in probability,}
\]
where $c$ is a constant whose value is numerically estimated to be about $0.7632...$ 
(Crisanti and Rizzo \cite{CrRi02}) and known to be bounded by
$\sqrt{2/\pi} \approx 0.797885...$ (Guerra \cite{Gue03}). 

% NEED TO CHECK

In this paper we are interested in locally optimal solutions. An
important reason of why local optima are worth considering is because
local optima may be computed quickly by simple greedy algorithms,
see \cite{AnBuPeWe16} and subsection \ref{sub:maxcut} below. We show that the expected number of local optima grows exponentially 
and we establish the rate of growth. Also, we examine the conditional distribution
of $H(\sigma) n^{-3/2}$ given that $\sigma$ is locally optimal. We prove that the
distribution is concentrated on an interval of exponentially small
width and determine the location.

\subsection{Results}

% NEED TO CHECK
In order to state the
main result of the paper, we need a few definitions.

Let $\Phi(\lambda)=\PROB\{N\le \lambda\}$ be the distribution function
of a standard normal random variable $N$ and introduce
$\phi(\lambda)= \log(2\Phi(\lambda))$. For $x\geq \sqrt{2/\pi}$, we let $\mu^*(x)$ denote the following Fenchel-L\'{e}gendre transform:
\[\mu^*(x) := \sup_{\lambda\geq 0}\,\left(\lambda\,x - \frac{\lambda^2}{2} - \phi(\lambda)\right).\]
Lemma \ref{lem:LDP} below shows that $\mu^*:[\sqrt{2/\pi},+\infty)\to \R$ is well defined. Lemma \ref{lem:estimatesforproof} shows that the mapping
\[x\geq \sqrt{\frac{2}{\pi}}\mapsto R(x):= \frac{x^2}{4} - \mu^*(x)\]
is strictly concave and achieves its global maximum at $x=v_*>\sqrt{\frac{2}{\pi}}$. We let $\alpha^*=R(v^*)>0$ denote the maximum value of $R$. 

\begin{theorem}
\label{thm:main}We have that, as $n\to +\infty$, for any choice of $\sigma=\sigma(n)\in \{-1,+1\}^n$,
\[
  \lim_{n\to +\infty}\frac{1}{n} \log \PROB\left\{ \sigma \ \text{is
     locally optimal} \right\} = \alpha^* - \log 2 .
\]
Moreover, there exists constants $\epsilon_0>0$, $L>0$ and $n_0$ such that, for $0<\epsilon <\epsilon_0$ and $n\geq n_0$,
\[\PROB\left\{\left. -\frac{v^*}{2} - \epsilon \le n^{-3/2} H(\sigma) \le -\frac{v^*}{2} + \epsilon \,\right| \,\sigma \text{ is locally optimal} 
\right\}  \geq 1- \exp\left(L\,\sqrt{n} -\epsilon^2\,n\right).\]
\end{theorem}

The values of the constants are numerically evaluated to be $\alpha^*\approx 0.199$ and $v^*/2\approx
0.506$. Since the global minimum of $H(\sigma)$ is about $- 0.763\,
n^{3/2}$, the typical value a local optimum $-0.506\,n^{3/2}$
comes fairly close. 

Also note that  Proposition \ref{prop:alphabounds} below implies that $\alpha^*$ is
between $1/(2\pi) \approx 0.1591\ldots$ and $2/(3\pi) \approx 0.2122\ldots$.

\subsection{Local minima, greedy algorithms and MaxCut}\label{sub:maxcut} 

Our problem is related to finding a local optimum of weighted MaxCut on the complete graph, which was recently studied in Angel, Bubeck, Peres, and Wei \cite{AnBuPeWe16}. Given $S\subset [n]$, we denote the value of the {\em cut} $(S,[n]\backslash S)$ as
\[{\rm Cut}(S,[n]\backslash S):= \sum_{i\in S}\sum_{j\in [n]\backslash S}\,W_{i,j}.\]
Note that there is a correspondence between cuts $(S,[n]\backslash S)$ and spin configurations $\sigma_S$ with:
\[\sigma_{S,i}:= 2\one_{i\in S} - 1 \,\,(i\in[n]).\]
\[{\rm Cut}(S,[n]\backslash S) = \frac{- H(\sigma_S) + \sum_{1\leq i<j\leq n}W_{ij}}{2}.\]
In particular, what Angel et al. call locally optimal cuts correspond exactly to our notion of local minimum. 

Starting from a given $\sigma$, do a sequence of local ``greedy moves"~ -- i.e. single spin flips that decrease energy -- until no more such moves are available. The main result of  \cite{AnBuPeWe16} is that this process ends at a local minimum after a polynomial number of moves. Unfortunately, it is not clear that the {\em distribution} of the value of this local minimum is similar to the one we study in Theorem \ref{thm:main}.

\section{The probability of local optimality}

% WORKING ON THIS

In this section we take the first and crucial step to prove Theorem \ref{thm:main}. For any fixed spin configuration $\sigma\in \{-1,+1\}^n$, we establish an integral formula for the probability that $\sigma$ is locally optimal. 

Define:
\begin{equation}\label{eq:defZi} Z_i(\sigma):= \frac{H(\sigma^{(i)}) - H(\sigma)}{2} = - \sum_{j\in [n]\backslash i}\sigma_i\,\sigma_j\,W_{i,j} \,\,(i\in[n]).\end{equation} 
Note that
\begin{equation}\label{eq:locmincondition} \sigma\mbox{ is a local minimum }\Leftrightarrow \forall i\in[n],\, Z_i(\sigma)\geq 0.\end{equation} 
Moreover, 
\begin{equation}\label{eq:locminenergy}-H(\sigma) =  \frac{1}{2}\sum_{i=1}^n\,\sum_{j\in [n]\backslash i}-\sigma_i\,\sigma_j\,W_{i,j} = \frac{\sum_{i=1}^nZ_i(\sigma)}{2}.\end{equation} 

Since $\sigma$ is fixed, we will write $Z_i$ instead of $Z_i(\sigma)$ most of the time. 

A key point in our calculations is that the random vector
\[Z=(Z_1,Z_2,\dots,Z_n)^T\] is a multivariate normal vector with zero mean and
covariance matrix $C = (C_{i,j})_{n\times n}$ such that $C_{i,i}= n-1$
for all $i\in [n]$ and $C_{i,j}=1$ for all $i\neq j$.
In other words,
\[
   C= (n-2) \Id_n + \one_n \one_n^T,
\]
where $\Id_n$ is the $n\times n$ identity matrix and $\one_n=(1,1,\ldots,1)^T$ is
the column vector with $1$ in each component. 

Clearly, the eigenvalues of $C$
are $2n-2$ with multiplicity $1$ and $n-2$ with multiplicity $n-1$, and therefore
$\det(C)=(2n-2)(n-2)^{n-1}$.

One may use the
Sherman-Morrison formula to invert $C$ and obtain
\[
   C^{-1} = \frac{1}{n-2}\left( \Id_n - \frac{1}{2n-2} \one_n \one_n^T\right)~,
\]
and therefore
\begin{eqnarray*}
\lefteqn{
\PROB\{ \sigma \ \text{is locally optimal} \}    } \\
& = &
\frac{1}{(2\pi)^{n/2} \det(\Sigma)^{1/2}}  \int_{[0,\infty)^n} \exp\left(\frac{-x^T C^{-1} x}{2}\right) dx \\
& = &
\frac{1}{(2\pi)^{n/2} (2n-2)^{1/2} (n-2)^{(n-1)/2}} \int_{[0,\infty)^n} \exp\left(\frac{- \|x\|_2^2}{2(n-2)} + \frac{ \|x\|_1^2}{2(n-2)(2n-2)}\right) dx
\\
& = &
2^{-n} \frac{1}{(2\pi)^{n/2} (2n-2)^{1/2} (n-2)^{(n-1)/2}}\int_{\R^n} \exp\left(\frac{- \|x\|_2^2}{2(n-2)} + \frac{ \|x\|_1^2}{2(n-2)(2n-2)}\right) dx~.
\end{eqnarray*}

% We may get a simple upper bound by using $\|x\|_1 \le \sqrt{n} \|x\|_2$. We obtain that the expression above is at most
% \begin{eqnarray*}
% 2^{-n} \frac{1}{(2\pi)^{n/2} (2n-2)^{1/2} (n-2)^{n-1/2}} \int_{\R^n} \exp\left(\frac{- \|x\|_2^2}{2(2n-2)} \right) dx \\
%  = 
% 2^{-n}  \frac{(2n-2)^{n/2}}{(2n-2)^{1/2} (n-2)^{n-1/2}} \sim 2^{-n/2} \sqrt{e/2}~,
% \end{eqnarray*}
% where we used the fact that
% \[
% \frac{1}{(2\pi)^{n/2} (2n-2)^{n/2}} \int_{\R^n} \exp\left(\frac{- \|x\|_2^2}{2(2n-2)} \right) dx = 1~.
% \]
We may rewrite this as:
\[
\PROB\{ \sigma \ \text{is locally optimal} \} = 2^{-n} \sqrt{\frac{n-2}{2n-2}} \;\EXP \exp\left( \frac{\|N\|_1^2}{4(n-1)} \right)
\]
where $N$ is a vector of $n$ independent standard normal random variables. 

In what follows, we derive some simple upper and lower bounds for the integral above. 

\begin{lemma}
If $N$ is a vector of independent standard normal random variables, then for all $\lambda >0$, 
\[
\lambda \EXP \|N\|_1^2 \le   \log \EXP \exp\left( \lambda \|N\|_1^2\right) 
\le \lambda \EXP \|N\|_1^2 \left(1+ \frac{n\lambda}{(1-n\lambda)} \right)~.
\]
\end{lemma}

\begin{proof}
The inequality on the left-hand side is obvious from Jensen's inequality.
To prove the right-hand side, we use the Gaussian logarithmic Sobolev inequality. 
In particular, writing $f(x)=\|x\|_1^2$ and $F(\lambda)= \EXP \exp\left( \lambda f(N)\right)$,
the inequality on page 126 of Boucheron, Lugosi, and Massart \cite{BoLuMa13} asserts
that
\[
    \lambda F'(\lambda) - F(\lambda)\log F(\lambda) \le \frac{\lambda^2}{2}
       \EXP\left[ e^{\lambda f(N)} \|\nabla f(N) \|^2\right]~.
\]
Since $\|\nabla f(N)\|^2=4n\|N\|_1^2$, we obtain the differential inequality
\[
    \lambda F'(\lambda) - F(\lambda)\log F(\lambda) \le 2n\lambda^2 F'(\lambda)~.
%       \EXP\left[ e^{\lambda f(X)} \|\nabla f(X) \|^2\right]~.
\]
This inequality has the same form as the one at the top of page 191 of \cite{BoLuMa13}  with $a=2n$ and $b=0$
and Theorem 6.19 implies the result above.
\end{proof}

Since
\[
\EXP  \|N\|_1^2 = n+ n(n-1) \frac{2}{\pi}~,
\]
we get
\[
\PROB\{ \sigma \ \text{is locally optimal} \} \ge 2^{-n} \sqrt{\frac{n-2}{2n-2}} 
\exp\left( n/(4(n-1)) + \frac{n}{2\pi} \right)
\]
and 
\[
\PROB\{ \sigma \ \text{is locally optimal} \} \le 2^{-n} \sqrt{\frac{n-2}{2n-2}} 
\exp\left(  \left(n/(4(n-1)) + \frac{n}{2\pi} \right) \frac{4n-1}{3n-1} \right) 
\]
Summarizing, we obtain the following bounds
\begin{proposition}
\label{prop:alphabounds}
For all spin configurations $\sigma\in \{-1,1\}^n$,
\[
\frac{1}{2\pi} -\log 2 -O(1/n)  \le   \frac{1}{n} \log \PROB\{ \sigma \ \text{is locally optimal} \} \le 
\frac{2}{3\pi} -\log 2 +O(1/n)
%\frac{- \log 2}{2} + O(1/n)
\]
\end{proposition}

%The constants on the two sides are about $-0.534$ and $-0.4809$.
%(The old upper bound was $-\log 2/2 \approx -0.3465$.)

In the next section we take a closer look at the integral expression
of the probability of local optimality. In fact, we prove that
$(1/n)\log \PROB\{ \sigma \ \text{is locally optimal} \}$ converges to
$\alpha^*-\log 2$ defined in the introduction.

\section{The value of local optima}

In this section we study, for any fixed $\sigma\in \{-1,+1\}^n$ and $\Delta
>0$, the joint probability
\[
   \PROB\left\{ \sigma \ \text{is locally optimal},\;n^{-3/2} H(\sigma)  \le - \Delta
     \right\}~.
\]
We let $Z=(Z_i)_{i=1}^n$ with $Z_i = Z_i(\sigma)$ as in the previous section. Recall from equations (\ref{eq:locmincondition}) and (\ref{eq:locminenergy}) that
\[\mbox{$\sigma$ is locally optimal}\Leftrightarrow \forall i\in [n],\, Z_i\geq 0\]
and
\[-\frac{H(\sigma)}{n^{3/2}} = \frac{1}{2n^{3/2}}\sum_{i=1}^nZ_i.\]
Therefore, we may follow the calculations in the previous section and obtain:
\begin{eqnarray*}
\lefteqn{
   \PROB\left\{ \sigma \ \text{is locally optimal}, n^{-3/2} H(\sigma)  \le - \Delta \right\}
} \\
& = &
\PROB\left\{ (\cap_{i=1}^n\{Z_i\geq 0\})\bigcap \left\{\sum_{i=1}^n Z_i \ge 2\Delta n^{3/2}\right\}\right\}   \\
& = &
\frac{1}{(2\pi)^{n/2} \det(C)^{1/2}}  \int\limits_{[0,\infty)^n\cap
  \{x:\sum_i x_i \ge 2\Delta n^{3/2}\}} \exp\left(\frac{-x^T C^{-1} x}{2}\right) dx \\
& = &
\frac{1}{(2\pi)^{n/2} (2n-2)^{1/2} (n-2)^{(n-1)/2}} \int\limits_{[0,\infty)^n\cap
  \{x:\sum_i x_i \ge 2\Delta n^{3/2}\}} \exp\left(\frac{- \|x\|_2^2}{2(n-2)} + \frac{ \|x\|_1^2}{2(n-2)(2n-2)}\right) dx
\\
& = &
2^{-n} \frac{1}{(2\pi)^{n/2} (2n-2)^{1/2}
  (n-2)^{(n-1)/2}}\int\limits_{\{x:\|x\|_1\ge 2\Delta n^{3/2}\}} \exp\left(\frac{- \|x\|_2^2}{2(n-2)} + \frac{ \|x\|_1^2}{2(n-2)(2n-2)}\right) dx~.
\end{eqnarray*}
Thus, by a change of variables, we get
\begin{eqnarray*}
\lefteqn{
\PROB\left\{ \sigma \ \text{is locally optimal},  n^{-3/2} H(\sigma)  \le - \Delta\right\}    } \\
& =  & 2^{-n} \sqrt{\frac{n-2}{2n-2}} 
\EXP \left[ \IND{\|N\|_1\ge 2\Delta n^{3/2}/\sqrt{n-2}}  \exp\left(
    \frac{\|N\|_1^2}{4(n-1)} \right) \right]~,
\end{eqnarray*}
where $N$ is a vector of independent standard normal random variables. 

We deduce the following proposition.

\begin{proposition}\label{prop:conditional}We have that, for all $\sigma\in \{-1,+1\}^n$,
\begin{equation}
\label{eq:key}
\PROB\left\{\left. -\frac{H(\sigma)}{n^{3/2}}\leq \Delta \right| \sigma \text{ is locally optimal} 
\right\}  = 
\frac{
\EXP \left[ \IND{\|N\|_1\ge 2\Delta n^{3/2}/\sqrt{n-2}}  \exp\left(
    \frac{\|N\|_1^2}{4(n-1)} \right) \right]}
{\EXP \exp\left(
    \frac{\|N\|_1^2}{4(n-1)} \right)}~.
\end{equation}\end{proposition}

\section{Approximating the integral}

In order to establish convergence of the exponent
$(1/n) \log \PROB\{ \sigma \ \text{is locally optimal} \}$ and also the
``typical''~value of the energy, we
need to understand the behavior of the numerator and the denominator 
of the key equation \eqref{eq:key}. 

The main idea is to obtain a Laplace-type approximation to the integral. Make the approximation
\[\EXP \left[ \exp\left(
    \frac{\|N\|_1^2}{4(n-1)} \right) \right]\approx \EXP\left[ \exp\left(
    \frac{\|N\|_1^2}{4n} \right) \right].\]
Observe that 
\[\frac{\|N\|_1}{n} = \frac{1}{n}\sum_{i=1}^n|N_i|\]
is an average of i.i.d. random variables expectation $\sqrt{2/\pi}$ and light tails. Therefore, it satisfies a Large Deviations Principle with a rate function $\mu^*(x)$:
\[\PROB\{\|N\|_1\geq nx\}\approx e^{-\mu^*(x)\,n}.\]
Readers familiar with Varadhan's Lemma (see e.g. \cite[page 32]{DenHollander_LD}) should expect that, as $n\to +\infty$,
\[\frac{1}{n}\log\,\EXP\left[ \exp\left(
    \frac{\|N\|_1^2}{4n} \right) \right] = \frac{1}{n}\log\,\EXP\left[ \exp\left(
    n\,\frac{(\|N\|_1/n)^2}{4} \right) \right] \to \sup_{v}\,\frac{v^2}{4} - \mu_*(v).\]
In fact, the intuition behind the Lemma is that most of the ``mass"~of the expectation concentrates around $\|N\|_1\sim v_*\,n$, where $v_*$ achieves the above supremum. This means that the conditional measure described in Proposition \ref{prop:conditional} should concentrate around $v^*/2$.

Our calculations confirm this reasoning. The usual statement of Varadhan's Lemma does not apply directly because $\|N\|_1^2/4$ is an unbounded function of $\|N\|_1$. Another minor technicality is that the function $\|N\|_1$ is divided by $(n-1)$ instead of $n$. In what follows we have opted for a self-contained approach  to our estimates, which gives quantitative bounds. This section collects the corresponding technical estimates. We finish the proof of Theorem \ref{thm:main} in the next section.

The next Lemma is a quantitative version of the large deviations principle (or Cram\'{e}r's Theorem) for $\|N\|_1$.

\begin{lemma}\label{lem:LDP}For $x\geq \sqrt{2/\pi}$, define $\mu^*(x)$ as in the introduction. Let $N=(N_1,\dots,N_n)$ be a vector of i.i.d. standard normal coordinates. 
Then:
\[\PROB\,\{\|N\|_1\geq nx\} = e^{-(\mu^*(x)+r_n(x))\,n}\]
with
\[0\leq r_n(x)\leq \kappa\,\left(\frac{x - \sqrt{2/\pi}}{\sqrt{n}} + \frac{1}{n}\right)\]
for some $\kappa>0$ independent of $x$ and $n$. Moreover, $\mu^*$ is smooth and $\mu^*(\sqrt{2/\pi})=0$.\end{lemma} 
\begin{proof}This follows directly from Lemmas \ref{lem:absnormal}, \ref{lem:lambdastar} and \ref{lem:chernoff} in subsection \ref{sec:technical}.\end{proof}

We will use this Lemma to estimate expectations of the form: 
\[\EXP\left[ \exp\left(
    \frac{c\|N\|_1^2}{2n} \right)\IND{\|N\|_1\leq an}\right]\mbox{ and }\EXP\left[ \exp\left(
    \frac{c\|N\|_1^2}{2n} \right)\IND{\|N\|_1\geq bn}\right].\]

The function $R_c$ defined below naturally shows up in our estimates. 
\begin{equation}\label{eq:defRc}R_c(x):=\frac{cx^2}{2} - \mu^*(x).\,\,\,(x\geq \sqrt{2/\pi}).\end{equation}

\begin{lemma}\label{lem:followsfromLaplace}For $a\geq \sqrt{2/\pi}$, $c\geq 0$
\[\EXP\left[ \exp\left(
    \frac{c\|N\|_1^2}{2n} \right)\IND{\|N\|_1\leq an}\right]  = (I) + (II),\]
    where 
    \[1\leq (I)\leq \exp\left(n\,R_c(\sqrt{2/\pi})\right)\]
    and
    \[ (II) = cn\int_{\sqrt{2/\pi}}^{a} x\,\exp(n\,(R_c(x)-r_{n}(x)))\,dx.\]
with $r_n(x)$ is as in Lemma \ref{lem:LDP}. For $b\geq \sqrt{2/\pi}$,
\begin{eqnarray*}\EXP\left[ \exp\left(
    \frac{c\|N\|_1^2}{2n} \right)\IND{\|N\|_1\geq bn}\right] &=&  \exp\left\{n\,(R_c(b) -r_n(b))\right\} \\ & & + cn\int_{b}^{+\infty} x\,\exp\left\{n\,(R_c(x)-r_n(x))\right\}\,dx.\end{eqnarray*}\end{lemma}
\begin{proof}Let $\phi_{c,n}(x) = e^{cn\,x^2/2}$. Note that:
\[ \IND{\|N\|_1\leq a\,n}  \exp\left(\frac{cn\,\|N\|_1^2}{2n}\right) = \phi_{c,n}\left(\frac{\|N\|_1}{n}\right)\,\IND{\frac{\|N\|_1}{n}\leq a}.\]
 We may compute the expectation of this expression as follows.
 \begin{eqnarray*}\EXP\left[\IND{\|N\|_1\leq a\,n}  \exp\left(
    \frac{cn\,\|N\|_1^2}{2n}\right)\right] &=& 1 + \int_{0}^a\,\phi'_{c,n}(x)\,\PROB\left\{\frac{\|N\|_1}{n}\geq x\right\}\,dx\\
	&=&  1 + cn \int_0^a\,\exp\left(\frac{cn\,x^2}{2}\right)\,\PROB\left\{\frac{\|N\|_1}{n}\geq x\right\}\,dx.\end{eqnarray*}
We split the above integral in two parts. 
\begin{eqnarray*}(I) &=& 1 + cn \int_0^{\sqrt{2/\pi}}\,x\,\exp\left(\frac{cn\,x^2}{2}\right)\,\PROB\left\{\frac{\|N\|_1}{n}\geq x\right\}\,dx\\
(II) &=& cn\,\int_{\sqrt{2/\pi}}^{a} \,x\,\exp\left(\frac{cn\,x^2}{2}\right)\,\PROB\left\{\frac{\|N\|_1}{n}\geq x\right\}\,dx.\end{eqnarray*}
For part (I), we bound the probability in the integral by $1$, and obtain:
\[1\leq (I)\leq 1 + cn \int_0^{\sqrt{2/\pi}}\,x\,\exp\left(\frac{cn\,x^2}{2}\right)\,dx \leq \exp\left(\frac{cn\,x^2}{2}\right)|_{x=\sqrt{\frac{2}{\pi}}} =  \exp\left(n\,R_c(\sqrt{2/\pi})\right)\]
because $\mu^*(\sqrt{2/\pi})=0$.
Term (II) may be evaluated using the estimate from Lemma \ref{lem:LDP}.
\[ (II) = cn\int_{\sqrt{2/\pi}}^{a} x\,\exp\left(\frac{cn\,x^2}{2} - n\,\mu^*(x) - n\,r_{n}(x)\right)\,dx,\]
which has the desired form because 
\[\frac{cn\,x^2}{2} - n\,\mu^*(x) = n\,R_c(x).\]

Similarly, \[ \IND{\|N\|_1\geq b\,n}  \exp\left(\frac{cn\,\|N\|_1^2}{2n}\right) = \phi_{c,n}\left(\frac{\|N\|_1}{n}\right)\,\IND{\frac{\|N\|_1}{n}\geq b},\]
and we finish the proof via the identity \begin{eqnarray*}\EXP\left[\IND{\|N\|_1\geq b\,n}  \exp\left(
    \frac{cn\,\|N\|_1^2}{2n}\right)\right] &=& \phi_{c,n}(b)\,\PROB\left\{\frac{\|N\|_1}{n}\geq b\right\} \\ & & + \int_{b}^{+\infty}\,\phi'_{c,n}(x)\,\PROB\left\{\frac{\|N\|_1}{n}\geq x\right\}\,dx\end{eqnarray*}
    and using the bounds in Lemma \ref{lem:LDP} (which are valid for all $x\geq b\geq \sqrt{2/\pi}$).\end{proof}

\section{Proof of main Theorem}

The previous section shows that, in order to estimate the expectations in Lemma \ref{lem:followsfromLaplace}, we need to understand the function $R_c$. The case of interest for us is when $c=n/2\,(n-1)$, which is when we recover the expectations in (\ref{eq:key}). Since $c$ varies with $n$, we will consider instead:
\begin{equation}\label{eq:defR14}R(x) = R_{\frac{1}{2}}(x):= \frac{x^2}{4} - \mu^*(x)\,\,(x\geq \sqrt{2/\pi}).\end{equation}
and note that 
\begin{equation}\label{eq:wrongRc}R(x)\leq R_c(x)\leq R(x) + (2c-1)\,\frac{x^2}{4}.\end{equation}

The next Lemma contains some information on $R(x)$.  
  
\begin{lemma}\label{lem:estimatesforproof}Let $x\geq \sqrt{2/\pi}$. Define $R$ as in equation (\ref{eq:defR14}) and $\mu^*$ as in Lemma \ref{lem:LDP}. Then there exists a unique $x=v^*>\sqrt{2/\pi}$ that maximizes $R(x)$ over $x\geq \sqrt{2/\pi}$. Leting $\alpha^*:=R(v^*)$ denote the value of the maximum, for any $x\geq \sqrt{2/\pi}$, there exists $\theta(x)\in [1/4,10]$ with:
\[R(x) - \alpha^*= -\theta(x)\,(x-v^*)^2.\]\end{lemma}
\begin{proof}See subsection \ref{sub:estimatesforproof}.\end{proof}

We can now obtain good upper and lower estimates on the integral expressions in Lemma \ref{lem:followsfromLaplace} and finish the proof of the main Theorem.

\begin{proof}[of Theorem \ref{thm:main}] In this proof we assume $n\geq 100$ for simplicity. We will use the notation $L>0$ to denote the value of a constant independent of $n$ that may change from line to line. Finally, we set 
\[c=c_n := \frac{n}{2\,(n-1)} = 1 + \frac{1}{2\,(n-1)}.\] Lemma \ref{lem:estimatesforproof} and (\ref{eq:wrongRc}) give:
\begin{equation}\label{eq:concavity}\forall x\geq \sqrt{\frac{2}{\pi}}\,:\,R_c(x)-\alpha^* \in\left[- 10\,(x-v^*)^2 ,- \frac{1}{6}\,(x-v^*)^2 + \frac{x^2}{(n-1)}\right].\end{equation}
We will now apply this to estimate expectations to the left of $v^*$. That is, we consider:
\[\EXP\left[ \exp\left(
    \frac{c\|N\|_1^2}{2n} \right)\IND{\|N\|_1\leq an}\right], \,\sqrt{\frac{2}{\pi}}\leq a\leq v^*.\]
In this range $|a-v^*|$ is uniformly bounded, so $x^2\leq L$ and
\[\forall \sqrt{\frac{2}{\pi}}\leq x\leq v^*\,:\,0\leq r_n(x)\leq \frac{L}{\sqrt{n}}.\]
Combining Lemma \ref{lem:followsfromLaplace} with $c\leq 1$ and (\ref{eq:concavity}), we obtain:
\begin{eqnarray*}\frac{\EXP\left[ \exp\left(
    \frac{c\|N\|_1^2}{2n} \right)\IND{\|N\|_1\leq an}\right]}{\exp(n\,\alpha^*)}&\leq& \exp(n\,(R_c(\sqrt{2/\pi})-\alpha^*)) \\ & & + cn \int_{\sqrt{2/\pi}}^a x\,\exp\left(n\,(R_c(x) - \alpha^*)\right)\,dx\\ 
    \\ &\leq & \exp\left(L - \frac{(v^*-\sqrt{2/\pi})^2}{4}\,n\right) \\ & & + n \int_{\sqrt{2/\pi}}^a x\,\exp\left(L + n\,\frac{(v^*-x)^2}{6}\right)\,dx\\ 
    &\leq & L\,(1+cn)\,\exp\left(L- \frac{(a-v^*)^2\,n}{4}\right)\\ &\leq & \exp\left(L\log n\, - \frac{(a-v^*)^2\,n}{4}\right).\end{eqnarray*}
    At the same time, 
\begin{eqnarray*}   \frac{ \EXP\left[ \exp\left(
    \frac{c\|N\|_1^2}{2n} \right)\IND{\|N\|_1\leq v^*n}\right]}{{\exp(n\,\alpha^*)}} &\geq &  \exp\left(-L\sqrt{n}\right)\,\int_{v^*-\frac{1}{n}}^{v^*} x\,\exp\left(n\,(R_c(x) - \alpha^*)\right)\,dx
    \\  &\geq &  \frac{1}{n}\,\left(v^*-\frac{1}{n}\right)\,\frac{\exp\left(-L\sqrt{n}\, - 10\,(1/n)^2\right)}{n}\\ &\geq &  \exp(-L\sqrt{n}).\end{eqnarray*}
 
For bounding the expectation for $b\geq v^*$, we cannot simply use $x^2\leq L$ and $r_n(x)\leq L/\sqrt{n}$. However, note that 
\[- \frac{1}{6}\,(x-v^*)^2 + \frac{x^2}{4\,(n-1)}\leq \left\{\begin{array}{ll}- \frac{1}{5}\,(x-v^*)^2 + \frac{L}{\sqrt{n}} &\mbox{for $x\leq (n-1)^{1/4}$};\\
   - \frac{1}{6}\,(x-v^*)^2 + \frac{2(x-v^*)^2 + 2(v^*)^2}{(n-1)}\leq -\frac{1}{5}\,(x-v^*)^2 + \frac{L}{n} &\mbox{ for larger $x$}.\end{array}\right.\]

Also, recalling the expression for $r_n$ in Lemma \ref{lem:LDP}, 
\[0\leq r_n(x)\leq \kappa\left(\frac{x-\sqrt{2/\pi}}{\sqrt{n}} + \frac{1}{n} \right)\leq \frac{L}{\sqrt{n}} + \frac{L\,(x-v^*)}{\sqrt{n}}.\]  

This allows us to obtain, for $b\leq v^*+\epsilon_0$, 
   
\begin{eqnarray*}\frac{\EXP\left[ \exp\left(
    \frac{c\|N\|_1^2}{2n} \right)\IND{\|N\|_1\geq bn}\right]}{\exp(n\,\alpha^*)}&\leq& \left(L\sqrt{n} - \frac{(b-v^*)^2\,n}{4}\right);\\  \frac{ \EXP\left[ \exp\left(
    \frac{c\|N\|_1^2}{2n} \right)\IND{\|N\|_1\geq v^*n}\right]}{{\exp(n\,\alpha^*)}}&\geq & \exp(-L\sqrt{n}).\end{eqnarray*}

This leads to our main results. Indeed, if we apply the above bounds with $a=b=v^*$, we obtain that, as $n\to +\infty$
\begin{eqnarray*}\EXP\left[ \exp\left(
    \frac{c\|N\|_1^2}{2n}\right)\right] &=&  \EXP\left[ \exp\left(
    \frac{c\|N\|_1^2}{2n} \right)\IND{\|N\|_1\leq v^*n}\right] \\  & & +\EXP\left[ \exp\left(
    \frac{c\|N\|_1^2}{2n} \right)\IND{\|N\|_1\geq v^*n}\right] \\ & = & \exp\left(n\,\alpha^* \pm L\sqrt{n}\right).\end{eqnarray*}
    This implies the first statement in the Theorem via Proposition \ref{prop:alphabounds}.

Secondly, we apply Proposition \ref{prop:conditional} and obtain:
\begin{eqnarray*}\PROB\left\{ - H(\sigma)\leq - \frac{v^*}{2} - \epsilon \mid \sigma\mbox{ local optimum}\right\} &\leq &  \frac{
\EXP \left[ \IND{\|N\|_1\ge b\,n}  \exp\left(
    \frac{\|N\|_1^2}{4(n-1)} \right) \right]}
{\EXP \exp\left(
    \frac{\|N\|_1^2}{4(n-1)} \right)}~
 \\ & & \left(\mbox{ with }b = \left(v^* + 2\epsilon\right)\,\sqrt{\frac{n}{n-1}}\right)\\ &= & \exp\left(-L\sqrt{n} - \epsilon^2\,n\right),\end{eqnarray*}
and (for $\epsilon$  small enough, so that $a$ below is $\geq \sqrt{2/\pi}$):
\begin{eqnarray*}\PROB\left\{ - H(\sigma)\geq - \frac{v^*}{2} + \epsilon \mid \sigma\mbox{ local optimum}\right\} & = &  \frac{
\EXP \left[ \IND{\|N\|_1\le a\,n}  \exp\left(
    \frac{\|N\|_1^2}{4(n-1)} \right) \right]}
{\EXP \exp\left(
    \frac{\|N\|_1^2}{4(n-1)} \right)}~
 \\ & & \left(\mbox{ with }a = \left(v^* - 2\epsilon\right)\,\sqrt{\frac{n}{n-1}}\right)\\ &=& \exp\left(-L\sqrt{n} - \epsilon^2\,n\right).\end{eqnarray*}
   \end{proof}
 
\section{Auxiliary results}

\subsection{Lemmas on large deviations of $\|N\|_1$}\label{sec:technical}

The goal of this subsection is to prove a series of Lemmas that together imply Lemma \ref{lem:LDP}. We first find an expression for the Laplace transform of $|N(0,1)|$

\begin{lemma}
\label {lem:absnormal}
Let $N(0,1)$ be a standard normal random variable. For all $\lambda >0$,
\[
   \EXP e^{\lambda|N(0,1)|}= e^{\lambda^2/2 + \phi(\lambda)}~,
\]
where $\phi(\lambda)= \log(2\Phi(\lambda))$, with
$\Phi(\lambda)=\PROB\{N(0,1)\le \lambda\}$.
\end{lemma}
\begin{proof}
\begin{eqnarray*}
\EXP e^{\lambda|N(0,1)|} & = & \frac{2}{\sqrt{2\pi}} \int_0^\infty
e^{\lambda x - x^2/2}dx  \\
& = & 2e^{\lambda^2/2} \frac{1}{\sqrt{2\pi}}\int_0^\infty
e^{(x-\lambda)^2/2}dx \\
& = & 2e^{\lambda^2/2} \PROB\{N(0,1) > - \lambda\}~.
\end{eqnarray*}
\end{proof}

We will need to compute the large deviations rate function for $\sum_{i=1}^n|N_i|$, with $N_i$ i.i.d. standard normal. As usual, this is given by the Fenchel-L\'{e}gendre transform of $\log \EXP e^{\lambda|N(0,1)|}$:
\[\mu^*(x):= \sup_{\lambda\geq 0}\,\lambda x - \log \EXP e^{\lambda|N(0,1)|}.\]

The next lemma collects technical facts on $\mu^*$ and the value $\lambda=\lambdastar$ that achieves the minimum.

\begin{lemma}\label{lem:lambdastar}For each $x\geq \sqrt{2/\pi}$, there exists a unique $\lambda=\lambdastar(x)\geq 0$ such that
\[\lambda + \phi'(\lambda) = x.\]
Defining:
\[x\geq \sqrt{\frac{2}{\pi}}\mapsto \mu^*(x) := \lambdastar(x)\,x -\frac{\lambdastar(x)^2}{2} - \phi(\lambdastar(x)),\]
we have that, for each $x$ in the above range, $\mu^*(x)$ is the global maximum of 
\[\lambda\geq 0\mapsto \lambda\,x - \frac{\lambda^2}{2} - \phi(\lambda),\]
which is uniquely achieved at $\lambda=\lambdastar(x)$. We also have the following inequalities.
 \begin{enumerate}
 \item {\em Strict concavity.} For each $\lambda\geq 0$, $x\geq \sqrt{2/\pi}$,
\begin{equation}\label{eq:growtharoundlambdastar}\mu^*(x) - (\lambda\,x - \ln \EXP\,e^{\lambda |N(0,1)|}) \in \left[\frac{(\lambda-\lambdastar(x))^2}{40},\frac{(\lambda-\lambdastar(x))^2}{2}\right].\end{equation}
\item {\em Derivative bounds for $\lambdastar$:} \begin{equation}\label{eq:lambdastarstarderivative}1\leq  \lambdastar'(x) \leq 20\end{equation}
\end{enumerate}\end{lemma}
\begin{proof}By the previous Lemma,
\[\lambda + \phi'(\lambda) = \frac{d}{d\lambda}\,\log\,\EXP e^{\lambda|N(0,1)|},\]
which is a smooth function because $|N(0,1)|$ has a Gaussian-type tail. Using this ``lightness of the tail",~ one can differentiate under the the expectation and obtain:
\[ \phi'(0) = \frac{d}{d\lambda}\EXP \,e^{\lambda|N(0,1)|}\mid_{\lambda=0} = \EXP\,|N(0,1)| = \sqrt{\frac{2}{\pi}}.\]
Lemma \ref{lem:phidoubleprime}~below implies that 
\[-0.95\leq \phi''(\lambda) \leq 0.\]
Therefore
\begin{equation}\label{eq:convexMGF}\forall \lambda\geq 0,\, \frac{d}{d\lambda}\,(\lambda + \phi'(\lambda))\in [0.05,1].\end{equation}
In particular, $\lambda + \phi'(\lambda)$ is an increasing function that is equal to $\sqrt{2/\pi}$ at $\lambda=0$ and diverges when $\lambda\nearrow +\infty$. It follows that for all $x\geq \sqrt{2/\pi}$ there exists a unique $\lambda=\lambdastar(x)$ with $\lambdastar(x)+\phi'(\lambdastar(x))=x$, and $\lambdastar(\sqrt{2/\pi})=0$. The Implicit Function Theorem guarantees that $\lambdastar$ is smooth over $[\sqrt{2/\pi},+\infty)$ and 
\begin{equation}\label{eq:lambdastarderivativehere}\lambdastar'(x) = \frac{1}{\frac{d}{d\lambda}\,(\lambda + \phi'(\lambda))\mid_{\lambda=\lambdastar(x)}}\in [1,20].\end{equation}
Equation (\ref{eq:convexMGF}) above shows that \[\lambda\,x - \frac{\lambda^2}{2} - \phi(\lambda) = \lambda\,x - \ln \EXP\,e^{\lambda |N|} \] is a strictly concave function of $\lambda$ with second derivative 
\[-1\leq -\frac{d^2}{(d\lambda)^2}\,\left(\frac{\lambda^2}{2} + \phi(\lambda) \right)\leq- \frac{1}{20}.\]
Thus $\lambdastar(x)$, which is a critical point for this function, is its unique global maximum of $\lambda\,x - \ln \EXP\,e^{\lambda |N|}$. The value of the function at that point is precisely $\mu^*(x)$. 

Let us now prove the estimates in the Lemma. The strict concavity property in (\ref{eq:growtharoundlambdastar}) follows from expanding 
\[\lambda\,x - \ln \EXP\,e^{\lambda |N(0,1)|}\]
around the critical point $\lambda=\lambdastar(x)$ and applying a second-order Taylor expansion:
\begin{eqnarray*}\lambda\,x - \ln \EXP\,e^{\lambda |N(0,1)|} - \mu^*(x)  &=&  \frac{d}{d{\lambda}}\,(\lambda\,x - \ln \EXP\,e^{\lambda |N(0,1)|})_{\lambda=\lambdastar(x)} \\ 
& & + \frac{1}{2}\,\frac{d^2}{(d\lambda)^2}\,\left(\lambda\,x - \ln \EXP\,e^{\lambda |N(0,1)|} \right)_{\lambda=\tilde{\lambda}}\,(\lambda - \lambdastar(x))^2 \\ & & \mbox{ with }\tilde{\lambda}= (1-\alpha)\lambdastar(x) + \alpha \lambda,\, 0\leq \alpha\leq 1.\end{eqnarray*}
noting that the first derivative is $0$ and the second one is between $-1$ and $-1/20$. Finally, the derivative bound in item $2$ is proven in (\ref{eq:lambdastarderivativehere}).
\end{proof}

\begin{lemma}
\label{lem:chernoff}
Let $N=(N_1,\ldots,N_n)$ be a vector of $n$ independent standard
normal random variables. Let $x\ge \sqrt{2/\pi}$. 
Define $\mu^*(x)$ as in Lemma \ref{lem:lambdastar}. Then for all $n\ge 1$,
\[
    \frac{1}{n} \log \PROB\left\{ \|N\|_1 \ge nx \right\} = 
    -\mu^*(x) - r_n(x)~,
\]
where
\[0\leq r_n(x)\leq   \kappa\,\left(\frac{x - \sqrt{2/\pi}}{\sqrt{n}} + \frac{1}{n}\right) \]
for some universal $\kappa>0$ that is independent of $x\geq \sqrt{2/\pi}$ and $n\geq 1$.\end{lemma}

\begin{proof}
For any $\lambda>0$, the usual Cram\'{e}r-Chernoff trick may be combined with Lemma \ref{lem:lambdastar}~to obtain:
\[\frac{1}{n}\log\PROB\left\{ \|N\|_1 \ge nx \right\} \leq \inf_{\lambda\geq 0}(\log\,\EXP\,e^{\lambda|N|} - \lambda\,x) = - \mu^*(x).\]
It remains to give a lower bound for this probability. In order to get a non-asymptotic bound, 
we use the following lemma that appears in the fourth edition of Alon and
Spencer's book \cite[Theorem A.2.1]{AlSp16}.

\begin{lemma}
\label{lem:alsp}
Let $u,\lambda, \epsilon >0$ such that $\lambda >\epsilon$. Let $X$ be
a random variable such that the moment generating function $\EXP
e^{cX}$ exists for $c\le \lambda +\epsilon$. For any $a\in \R$, define
$g_a(c)= e^{-ac} \EXP e^{cX}$. Then 
\[
   \PROB\{ X \ge a-u\} \ge e^{-\lambda u}\left( g_a(\lambda) -
     e^{-\epsilon u} \left( g_a(\lambda+\epsilon) +
       g_a(\lambda-\epsilon) \right) \right)~. 
\]
\end{lemma}

Lemma \ref{lem:alsp} will be applied to the random variable $X=\|N\|_1$ with $\lambda=\lambdastar(a/n)$ and $a,u,\eps$ to be chosen. 

In the notation of that Lemma, for each $\lambda\geq 0$,
\[
   g_a(\lambda) = \exp\left(-n \mu_a(\lambda)\right)\mbox{ where }\mu_a(\lambda)=\left(\lambda\,(a/n) - \ln\,\EXP\,e^{\lambda\,|N|} \right).
\]
Using Lemma \ref{lem:lambdastar} to bound this expression, we obtain from (\ref{eq:growtharoundlambdastar}) that
\[
   \frac{g_a(\lambdastar(a/n)+ \epsilon)}{g_a(\lambdastar(a/n))}
    \le e^{n\epsilon^2/2}
\quad \text{and} \quad
   \frac{g_a(\lambdastar(a/n)- \epsilon)}{g_a(\lambdastar(a/n))}
    \le e^{n\epsilon^2/2}~.
\]
Moreover, $g_a(\lambdastar(a/n)) = e^{-n\mu^*(a/n)}$. So

\[
   \PROB\{ \|N\|_1 \ge a-u\} \ge e^{-\lambdastar(a/n)\,u}e^{-n\mu^*(a/n)}\,\left( 1 - 2
     e^{-\epsilon u + \frac{\eps^2n}{2}}  \right)~. 
\]
We now choose $\epsilon=\sqrt{2/n}$ and $u=\epsilon n/2+ 1/\epsilon=\sqrt{2n}$ to obtain:
\[ \PROB\{ \|N\|_1 \ge a-u\} \ge e^{-\lambdastar(a/n)\,\sqrt{2n}}e^{-n\mu^*(a/n)}\,\left(1 - \frac{2}{e}\right)~. \]
Letting $a=nx+\sqrt{2n} = n\,(x+\epsilon)$, we have that
\[\PROB\{\|N\|_1 \ge nx\} = e^{-\lambdastar(x+\epsilon)\,\sqrt{2n}}e^{-n\,\mu^*(x+\epsilon)}\,\left(1 - \frac{2}{e}\right).\]
Recall from Lemma \ref{lem:lambdastar} that $\lambdastar'(y)\leq 20\,(y-\sqrt{2/\pi})$ and $1\,(y-\sqrt{2/\pi})\leq (\mu^*)'(y)\leq 20\,(y-\sqrt{2/\pi})$ for all $y\geq \sqrt{2/\pi}$. Thus:
\[\lambdastar(x+\epsilon)\,\sqrt{2n}\leq 20\,(x + \epsilon -\sqrt{2/\pi})\,\sqrt{2n}\]
and
\[\mu^*(x+\epsilon)\leq \mu^*(x) + 20\,\epsilon\,(x+\epsilon-\sqrt{2/\pi}).\]
Recalling $\epsilon = \sqrt{2/n}$, we may plug theses estimates back in the lower bound for our probability and obtain the theorem. \end{proof}

\subsection{Estimates on the optimization problem}\label{sub:estimatesforproof}

In this section we prove Lemma \ref{lem:estimatesforproof}.

\begin{proof}[of Lemma \ref{lem:estimatesforproof}] We will use Lemma \ref{lem:lambdastar} several times in the proof. In particular, the properties of $\mu^*$ and $\lambdastar=(\mu^*)'$ in that proof will be used several times. 

We first argue that $x\mapsto R(x)$ is a strictly concave function of $x\geq \sqrt{2/\pi}$. To see this, we use Lemma \ref{lem:lambdastar} to obtain:
\[R''(x) = \frac{1}{2} - \lambdastar'(x).\]
By the same Lemma, we know $1\leq \lambdastar'(x)\leq 20$. So
\[ \forall x\geq \sqrt{\frac{2}{\pi}}\,:\, -20\leq R''_c(x) \leq -\frac{1}{2}.\]

We now argue that $R(x)$ is maximized at some  $x=v^*>\sqrt{2/\pi}$. To see this, notice that $\lambdastar(\sqrt{2/\pi})=0$, so the derivative of $R$ at $x=\sqrt{2/\pi}$ satisfies:
\[R'(x) = \frac{1}{2}\sqrt{2/\pi} - \lambdastar(\sqrt{2/\pi}) >0\]
We conclude that $R$ is increasing in an interval to the right of $\sqrt{2/\pi}$. At the same time, the second derivative of $R$ in $x$ is at most $-1$, so there exists a $s_+$ such that $R'(x)/\leq 0$ for $x\geq s_+$. So the maximum of $R$ in $x$ must be achieved at a point $v^*\in (\sqrt{2/\pi},s_+]$. In particular, $v^*$ is a critical point of $R$: $R'(v^*)=0$.

Now consider, \[\alpha^* := R(v^*) = \min_{x\geq \sqrt{2/\pi}}R(x).\]
By Taylor expansion, if $x\geq \sqrt{2/\pi}$,
\[R(x) = \alpha^* + R'(v^*)\,(x-v^*) + \frac{R''(v^* + \alpha\,(x-v^*))}{2}\,(x-v^*)^2\]
where $0\leq \alpha\leq 1$. The Theorem follows because $R'(v^*)=0$ and $R''\in [-20,-1/2]$.\end{proof}

\appendix \section{One more technical estimate}

\begin{lemma}
\label{lem:phidoubleprime}
Let $f(\lambda)=(2\pi)^{-1/2} e^{-\lambda^2/2}$ be the standard normal
density let $\Phi(\lambda) = \int_{-\infty}^\lambda f(x)dx$ be the
corresponding cumulative distribution function. Then for all $\lambda
\ge 0$,
\[
\frac{f'(\lambda)}{\Phi(\lambda)} -
\frac{f(\lambda)^2}{\Phi(\lambda)^2} > -0.95~.
\]
\end{lemma}

\begin{proof}
Note that $f'(\lambda)= -\lambda f(\lambda)$, so we need only prove
\[
    \sup_{\lambda \ge 0} \frac{f(\lambda)}{\Phi(\lambda)}\left(
      \lambda + \frac{f(\lambda)}{\Phi(\lambda)} \right) < 0.95~.
\]
We combine three inequalities, considering three ranges of the value
of $\lambda$, given by $[0,\lambda_1)$, $[\lambda_1,\lambda_2]$, and
$(\lambda_2,\infty)$,
where
\[
%   \lambda_1=\sqrt{\frac{\pi}{2}} - \sqrt{\frac{2}{\pi}} 
 \lambda_1= \frac{0.95 -\frac{2}{\pi}}{\sqrt{2/\pi}}
\approx 0.3927\ldots
%   0.4554\ldots
\quad \text{and} \quad
\lambda_2 = \sqrt{\log \frac{2/\pi}{0.95-\sqrt{2/(\pi e)}}   } \approx 
0.5584\ldots~.
%0.4582\ldots~.
\]
First, note that $f(\lambda)/\Phi(\lambda) \le \sqrt{2/\pi}$ since
$f/\Phi$ is a decreasing function. Thus,
\[
\frac{f(\lambda)}{\Phi(\lambda)}\left(
      \lambda + \frac{f(\lambda)}{\Phi(\lambda)} \right) 
\le 
\frac{2}{\pi}  + \lambda \sqrt{\frac{2}{\pi}} 
< 0.95 \quad \text{for $\lambda \in [0,\lambda_1)$.}
\]
Second, $\lambda e^{-\lambda^2/2} \le 1/\sqrt{e}$, so 
\[
\frac{f(\lambda)}{\Phi(\lambda)}\left(
      \lambda + \frac{f(\lambda)}{\Phi(\lambda)} \right) 
\le 
   2\lambda f(\lambda) + 2 f(\lambda)^2 \le 
  \sqrt{\frac{2}{\pi e}} + \frac{2}{\pi} e^{-\lambda^2} < 0.95
\quad \text{for $\lambda \in (\lambda_2,\infty)$.}
\]
Finally, since $\lambda e^{-\lambda^2/2}$ is increasing and
$e^{-\lambda^2/2}$ is decreasing on $[\lambda_1,\lambda_2]$,
on this interval we have
\[
\frac{f(\lambda)}{\Phi(\lambda)}\left(
      \lambda + \frac{f(\lambda)}{\Phi(\lambda)} \right) 
\le 2\lambda_2 f(\lambda_2) + 4 f^2(\lambda_1) \approx 
0.92685\ldots <0.95~.
%0.5355\ldots <1~.
\]
\end{proof}

\bibliographystyle{plain}
\bibliography{LocalOptima}

\end{document}